\documentclass[11pt]{article}%
\usepackage{amsmath}
\usepackage{amsfonts}
\usepackage{amssymb}
\usepackage{graphicx}%
\newtheorem{theorem}{Theorem}

\newtheorem{corollary}[theorem]{Corollary}

\newtheorem{lemma}[theorem]{Lemma}

\newenvironment{proof}[1][Proof]{\noindent\textbf{#1.} }{\ \rule{0.5em}{0.5em}}
\graphicspath{    {converted_graphics/}    {/}}
\begin{document}

\begin{center}
{\LARGE The periodic integral orbits of polynomial}

{\LARGE recursions with integer coefficients}

\medskip

Hassan Sedaghat
\footnote{Professor Emeritus of Mathematics, Virginia Commonwealth
University, Richmond, VA \ 23284, USA;
\par
Email: hsedagha@vcu.edu}

\end{center}

\medskip

\begin{abstract}
We show that polynomial recursions $x_{n+1}=x_{n}^{m}-k$ where $k,m$ are integers 
and $m$ is positive have no nontrivial periodic integral orbits for $m\geq3$. 
If $m=2$ then the recursion has integral two-cycles for infinitely many values 
of $k$ but no higher period orbits. We also show that these statements are true
for all quadratic recursions.
\end{abstract}

\bigskip

The quadratic recursion%
\begin{equation}
x_{n+1}=x_{n}^{2}+c \label{q1}%
\end{equation}
and its topological conjugate, the discrete logistic equation are iconic
examples of nonlinear dynamical systems. Starting with an initial value
$x_{0}$ with $n=0$ in (\ref{q1}) we may calculate the values of $x_{1}$,
$x_{2}$, etc recursively and generate a sequence $x_{n}$ that is considered a
(forward) \textit{orbit or solution} of (\ref{q1}).

If the parameter $c$ is a real or complex number then (\ref{q1}) can have a
wide variety of bounded orbits. For example, if $c=-2$ then (\ref{q1}) has
real periodic orbits (or cycles) of all possible periods in the interval
$[-2,2]$ depending on the initial value $x_{0}$ as well as certain bounded,
oscillating but non-periodic orbits that are called \textit{chaotic}; see,
e.g. \cite{L-Y}, \cite{HS}.

If $c$ is a rational number then the possible \textit{rational} orbits of
(\ref{q1}) are far more restricted. It is shown in \cite{W-R} that rational
fixed points (cycles of period 1), as well as rational cycles of period 2 may
occur but (in contrast to the case of real orbits) no cycles of greater
periods occur if $c$ is a rational number with odd denominator. In particular,
if $c$ is an integer (e.g. $c=-2$) then $l=1$ and we see that (\ref{q1}) has 
no integral cycles of period larger than 2. It is also shown in \cite{W-R} 
that rational cycles of period 3 occur for some rational values of $c$ with
even denominator. But the occurrence of this rational cycle of period 3 does 
not automatically imply the occurrence of rational cycles with other periods.

The non-existence proofs in \cite{W-R} use the properties of p-adic rationals.
But these results do not extend in an obvious way to the higher degree polynomial 
recursions like%
\begin{equation}
x_{n+1}=x_{n}^{m}+c \label{hd1}%
\end{equation}
where $m$ is an integer greater than 2, so it is not clear whether
(\ref{hd1}) with rational $c$ has rational periodic orbits of period 2 or
greater in the higher degree cases.

\medskip

In this paper we discuss the periodic \textit{integral} orbits of
(\ref{hd1}), i.e. orbits that are contained in the set $\mathbb{Z}$ of all
integers. Specifically, if $k$ is an integer then all orbits of%
\begin{equation}
x_{n+1}=x_{n}^{m}-k \label{hd}%
\end{equation}
with integer initial values are contained in $\mathbb{Z}$. For most integer
initial values these integral orbits are unbounded but the question that we
answer here is whether \textit{all} orbits, except for possible integer fixed
points of (\ref{hd}) are unbounded.

As we noted above, the answer is negative if $m=2$ but the periodic integral
orbits were scarce in this case. We show that if $m>2$ then there are no
periodic orbits of (\ref{hd}), except possibly for
one or two integer fixed points. The main idea is simply that the points of an
integral orbit do not come close to each other.

We also show that the results for $m=2$ are true for the general quadratic 
recursion
\[
x_{n+1}=ax_{n}^{2}+bx_{n}+c
\]
where $a,b,c$ are integers ($a$ nonzero) and determine the equations for the
fixed points and the two-cycles in terms of the coefficients $a,b,c$.

\section{There are no nontrivial periodic orbits if $m>2$}

We begin by recalling a few basic concepts and identifying exceptional and/or
trivial cases. The recursion in (\ref{hd}) can be written as%
\[
x_{n+1}=f(x_{n}),\quad f(x)=x^{m}-k
\]

A fixed point of the function $f(x)$, i.e. a solution of the equation
\[
f(x)=x
\]
is also called a \textit{constant solution} or a \textit{fixed point} of the
recursion (\ref{hd}). We also call it a \textit{trivial orbit}.

A \textit{periodic} \textit{orbit }or \textit{cycle} of (\ref{hd}) is a
sequence $r_{0},r_{1},r_{2},\ldots$ where there is a positive integer $p$ such
that
\begin{equation}
r_{n+p}=r_{n}\text{\quad for all }n\text{.} \label{p}%
\end{equation}

If $p$ is the \textit{smallest} positive integer for which (\ref{p}) is true
then $p$ is the \textit{period} of $r_{n}$; we also call a cycle of period $p$
a \textit{p-cycle} for short. Finally, a \textit{bounded orbit }of (\ref{hd})
is a bounded sequence $r_{0},r_{1},r_{2},\ldots$ that satisfies (\ref{hd}).

\medskip

Note that if $m=1$ then the recursion%
\begin{equation}
x_{n+1}=x_{n}-k \label{thd}%
\end{equation}
has the general solution%
\[
x_{n}=x_{0}-nk
\]

From this we conclude that every solution of (\ref{thd}) diverges to $\infty$
if $k<0$ and to $-\infty$ if $k>0$. If $k=0$ then every solution of
(\ref{thd}) is constant (every initial value is fixed). So in the rest of 
the paper we assume that $m\geq2$.

If $k=0$ then (\ref{hd}) has two integer fixed points 0 and 1 if $m$ is even
and three integer fixed points 0, $\pm1$ if $m$ is odd.

If $k\not =0$ and $m$ is odd then using Descartes rule of signs and the
intermediate value theorem we see that (\ref{hd}) has only one real fixed
point $\gamma$ that is positive if $k>0$ (i.e. $k\geq1$) and negative if $k<0$
(i.e. $k\leq-1$). Further, it is easy to see that all (real) orbits of
(\ref{hd}) diverge to $\infty$ if $x_{0}>\gamma$ and to $-\infty$ if
$x_{0}<\gamma$. In particular, (\ref{hd}) has no nontrivial, bounded integral
orbits if $m$ is odd.

\medskip

What about possible integer fixed points? An integer $j$ satisfies
$x=x^{m}-k$ if and only if
\[
k=j^{m}-j=j(j^{m-1}-1)
\]

It follows that for every integer $j$ the equation%
\[
x_{n+1}=x_{n}^{m}-j(j^{m-1}-1)
\]
has a fixed point at $j$. In particular, \textit{(\ref{hd}) has integer fixed
points for infinitely many values of} $k$.

For example, if $m=3$ and $k$ is the (even) number
\[
k=j(j^{2}-1)=j(j-1)(j+1)
\]
then (\ref{hd}) has a fixed point at $x=j$. Similarly, if $m=4$ and
\begin{equation}
k=j(j^{3}-1)=j(j-1)(j^{2}+j+1) \label{m4}%
\end{equation}
then (\ref{hd}) has a fixed point at $x=j$.

\medskip

\textit{In the rest of the paper we assume that }$m$\textit{ is even if its
value is not specified.}

\medskip

If $k=1$ then can quickly check that (\ref{hd}) has an integral cycle of 
period 2: $-1,0,-1,0,\ldots$ for 
\textit{every} even value of $m$. Further, the initial value $x_{0}=1$ leads to 
this cycle in one step. On the other hand, if $|x_{0}|\geq2$ then for all $m\geq2$
\begin{align*}
x_{1}  &  \geq2^{m}-1\geq3\\
x_{2}  &  \geq3^{m}-1\geq8\\
&  \vdots
\end{align*}
which is an increasing sequence of integers that rapidly diverges to $\infty$.
It follows that all other orbits are unbounded.

\medskip

We now consider the remaining cases where $k,m\geq2$ and $m$ is even.

\medskip

Now $f(x)$ has a minimum at 0 and two fixed points $\alpha$ and $\beta$
where
\[
\alpha<0<\beta
\]
and%
\begin{equation}
\alpha^{m}=\alpha+k,\qquad\beta^{m}=\beta+k \label{efp}%
\end{equation}

Note that $k>0$ if and only if $\beta>1$ by the right hand side equation 
above. Further,
\[
k=\beta^{m}-\beta=\beta(\beta^{m-1}-1)
\]
so $k>\beta$ if $\beta>2^{1/(m-1)}$. In particular, \textit{If }$\beta\geq
2$\textit{ then} $k>\beta$.

The next result shows that it is only necessary to consider orbits of (\ref{hd}) 
that start in $[-\beta,\beta]$ even though this interval is usually not invariant.

\begin{lemma}
\label{L1}For each $m\geq2$ if $|x_{0}|>\beta >1$ then the orbit generated by 
(\ref{hd}) is unbounded, eventually increasing to $\infty$.
\end{lemma}

\begin{proof}
First, note that $f(x)>x$ for all $x>\beta$ so that if $|x_{0}|>\beta$ then
\[
x_{1}=f(x_{0})>x_{0}%
\]

Further,
\begin{align*}
x_{1}  &  =x_{0}^{m}-k>\beta^{m}-k=\beta\\
x_{2}  &  =x_{1}^{m}-k>\beta^{m}-k=\beta\\
&  \vdots
\end{align*}
so by induction, $x_{n}>\beta$\ for every $n\geq1$. Now,
\begin{align*}
x_{n}-x_{n-1}  &  =x_{n-1}^{m}-x_{n-2}^{m}\\
&  =(x_{n-1}-x_{n-2})\sum_{i=1}^{m}x_{n-1}^{m-i}x_{n-2}^{i-1}\\
&  >(x_{n-1}-x_{n-2})\sum_{i=1}^{m}\beta^{m-i}\beta^{i-1}\\
&  =m\beta^{m-1}(x_{n-1}-x_{n-2})
\end{align*}

Doing the same calculation for $x_{n-1}-x_{n-2}$ then for $x_{n-2}-x_{n-3}$
and so on, we obtain by induction%
\[
x_{n}-x_{n-1}>(m\beta^{m-1})^{n-1}(x_{1}-x_{0})
\]

Therefore,%
\begin{align*}
x_{n}  &  =x_{0}+\sum_{i=1}^{n}(x_{i}-x_{i-1})\\
&  >x_{0}+\sum_{i=1}^{n}(m\beta^{m-1})^{i-1}(x_{1}-x_{0})\\
&  =x_{0}+(x_{1}-x_{0})\frac{(m\beta^{m-1})^{n}-1}{m\beta^{m-1}-1}%
\end{align*}

Due to the occurrence of the n-th power of $m\beta^{m-1}>m$ in the last
quantity it follows that $x_{n}\rightarrow\infty$ (exponentially fast) as
$n\rightarrow\infty$ if $x_{0}\not \in \lbrack-\beta,\beta]$.
\end{proof}

\medskip

It is worth a mention that $m\beta^{m-1}=f^{\prime}(\beta)$ is
the slope of the tangent line to the graph of $f(x)$ at $x=\beta$. We could
use this tangent line for an alternative proof but that was not necessary.

Also notice that the number $\beta$ is considerably smaller than $k$; for
instance, for $\beta\leq2$ the right hand side equation in (\ref{efp}) gives
\[
k=\beta(\beta^{m-1}-1)\leq2(2^{m-1}-1)=2^{m}-2
\]

If $m=4$ and $k=2$ then by (\ref{m4}) and the intermediate value theorem
$\alpha=-1$ and $1<\beta<2$. So $k>\beta$ and the interval 
$[-\beta,\beta]\subset\lbrack-2,2]$ contains the 3 integers $0$, $\pm1$.
A quick calculation shows that if $x_{0}=\pm1$ then $x_{1}=-1$ which is the 
fixed point, and further, if $x_{0}=0$ then $x_{1}=-2<-\beta$ so $x_{n}%
\rightarrow\infty$ as $n\rightarrow\infty$. So with $k=2$ the
only periodic integral orbit of (\ref{hd}) is the trivial one $x_{n}=-1$ for
all $n$. There is one more bounded integral orbit, namely the one that starts
at $x_{0}=1$.

Quick calculations show that (\ref{hd}) has precisely one integral periodic
orbit if $k=1$, and two fixed points if $k=0$ and no bounded solutions
(integer or not) if $k<0$ (i.e. $k\leq-1$) because
\[
x^{m}-k\geq x^{m}+1>x
\]
for all even integers $m$ and all real values of $x$. Now we determine what
happens when $k\geq3$.

\medskip

Considering orbits that start in $[-\beta,\beta]$, due to the y-axis symmetry
we need only check the integers in the interval $[0,\beta]$. If $x_{0}%
\in\lbrack0,\beta]$ then%
\[
|x_{1}|=|x_{0}^{m}-k|\leq\beta
\]
if and only if
\[
-\beta+k\leq x_{0}^{m}\leq\beta+k=\beta^{m}%
\]

Only the left hand side inequality poses a new restriction, namely,
\[
x_{0}\geq(k-\beta)^{1/m}%
\]

Let%
\[
\gamma=(k-\beta)^{1/m}%
\]
and note that
\begin{align*}
x_{0}  &  \in[\gamma,\beta]\Longrightarrow x_{1}\in\lbrack-\beta,\beta]\\
x_{0}  &  \in[0,\gamma)\Longrightarrow x_{1}<-\beta
\end{align*}

Thus it is necessary that the interval $[\gamma,\beta]$ contain an integer.
In this regard, the next lemma is important.

\begin{lemma}
If $m$ is even and larger than 2 then the length of the interval $[\gamma,\beta]$
is less than 1. Further, $\beta -\gamma\rightarrow 0$ for each such value of $m$
as $k\rightarrow\infty$.
\end{lemma}

\begin{proof}
Observe that%
\[
\beta^{m}-\gamma^{m}=k+\beta-(k-\beta)=2\beta
\]
which yields%
\[
\beta-\gamma=\frac{2\beta}{\sum_{i=1}^{m}\beta^{m-i}\gamma^{i-1}}<\frac
{2}{\beta^{m-2}}%
\]

Recall that $\beta>k^{1/m}$ so%
\begin{equation}
\beta-\gamma<\frac{2}{k^{(m-2)/m}} \label{bmg}%
\end{equation}

With $m>2$ the right hand side of the above inequality is less than 1 if%
\begin{align*}
k^{(m-2)/m}  &  >2\\
k  &  >2^{m/(m-2)}%
\end{align*}

The largest value of $m/(m-2)=1+2/(m-2)$ occurs at the smallest value of $m$,
i.e. $m=4$. Thus%
\[
2^{1+2/(m-2)}\leq4\quad\text{for }m=4,6,8,\ldots
\]

So if $k\geq4$ then the right hand side of (\ref{bmg}) is less than 1 for all
$m=4,6,8,\ldots$ and we obtain%
\[
\beta-\gamma<1
\]

Further, for each fixed value of $m$ (\ref{bmg}) also shows that $\beta
-\gamma\rightarrow0$ as $k\rightarrow\infty$. 
\end{proof}

\medskip

The above lemma in particular implies that $[\gamma,\beta]$ contains at most one 
integer. 

\begin{theorem}
The equation in (\ref{hd}) has no nontrivial integral periodic orbits for $m\geq3$.
\end{theorem}

\begin{proof}
We discussed the non-existence for all odd values of $m$ earlier, so now assume 
that $m$ is even and also for this theorem, $m\geq4$.

We first show that if $x_{0}\in(\gamma,\beta)$ then \textit{a non-constant
periodic orbit may exist only if }$x_{1}\in(-\beta,-\gamma)\cup(\gamma,\beta)$.

Note that since the x-intercept of $f(x)=x^{m}-k$ is $k^{1/m}%
\in\lbrack\gamma,\beta]$ and $f$ is increasing for $x>0$ it follows that $f$
maps $[k^{1/m},\beta]$ one-to-one onto $[0,\beta]$ with $f(k^{1/m})=0$.
Similarly, $f$ maps the interval $[\gamma,k^{1/m}]$ onto $[-\beta,0]$ with
$f(\gamma)=-\beta$. Since $x_{0}\in(\gamma,\beta)$ it follows that $x_{1}%
\in(-\beta,\beta)$.

In order that $x_{0}$ and $x_{1}$ be part of a periodic orbit it is necessary
that $x_{2}=f(x_{1})\in(-\beta,\beta)$ also. This is possible only if
$x_{1}\in(-\beta,-\gamma)\cup(\gamma,\beta)$ in which case $|x_{1}|\in
(\gamma,\beta)$.

Notice that \textit{an integral orbit of (\ref{hd}) cannot have a period
greater than 2} because the set $(-\beta,-\gamma)\cup(\gamma,\beta)$ contains
at most two integers.

If $x_{0},x_{1}$ form an integral orbit of period 2 for (\ref{hd}) with
$x_{0}\in(\gamma,\beta)$ then $x_{1}$ must be in the interval $(-\beta
,-\gamma)$. It follows that
\begin{equation}
x_{1}=-x_{0} \label{x10}%
\end{equation}

We also require that $x_{2}=x_{0}$ to close the cycle. Therefore,%
\begin{equation}
x_{0}=x_{2}=x_{1}^{m}-k=(-x_{0})^{m}-k=x_{1} \label{x01}%
\end{equation}
where the last equality holds since $m$ is even. The equalities (\ref{x10})
and (\ref{x01}) hold simultaneously if and only if $x_{0}=0$ which contradicts
our assumption about where $x_{0}$ is. Therefore, there can be no orbits of
period 2 for (\ref{hd}).

We have shown that if $k\geq4$ then the only
possible integral cycles of (\ref{hd}) are the fixed points. We still need
to examine the values of $k<4$, i.e. $k\leq3$. We have already checked the
solutions of (\ref{hd}) for $k\leq2$. Now, if $k=3$ then $\beta<2$ since $k$
is an increasing function of $\beta$ for $\beta\geq1$ and at $\beta=2$%
\[
k=2^{m}-2\geq2^{4}-2=14
\]

With $\beta<2$ the only integers in $[-\beta,\beta]$ are 0 and $\pm1$. With
$k=3$, if $x_{0}=0,\pm1$ then $x_{1}\leq-2<-\beta$ so $x_{n}\rightarrow\infty$
as $n\rightarrow\infty$. \ Further, the fixed points of $x^{m}-3$ are the
zeros of $x^{m}-x-3$ which by the intermediate value theorem are in the
intervals $(-2,-1)$ and $(1,2)$ for all $m=2,4,6,\ldots$ Since these fixed
points are not integers it follows that (\ref{hd}) has no periodic integral
orbits with $k=3$. This completes the proof of the theorem.
\end{proof}

\medskip

If $m=2$ then some of the steps in the above argument are invalid; in fact,
for $m=2$ it is the case that $\beta-\gamma\geq1$ for all $k$ and this opens
the way for the existence of 2-cycles. As shown in \cite{W-R} integral orbits
of period 2 indeed exist for (\ref{hd}), yet there are no such orbits with
greater periods. I prove this fact in the next section without using the
p-adic numbers.

\section{Periodic orbits of the quadratic recursion}

Many of the results of the previous section hold when $m=2$ but as the
next lemma shows, it is no longer the case that $\beta-\gamma<1$. In the case
$m=2$ the fixed points can be determined explicitly by solving the fixed point 
equation $f(x)=x$. This is the quadratic equation $x^{2}-x-k=0$ whose solution is%
\begin{equation}
\beta=\frac{1+\sqrt{1+4k}}{2} \label{bk}%
\end{equation}

This in turn gives an explicit formula for $\gamma=\sqrt{k-\beta}$; note that
$\beta$ is an increasing function of $k$ and a simple calculation shows the
same to be true for $\gamma$. Further, if $\beta$ is an integer then so is the
other fixed point $\alpha=1-\beta$. So, unlike the higher degree cases,
\textit{integer fixed points always occur in pairs when} $m=2$.

\begin{lemma}
\label{L3}For all $k\geq2$%
\begin{equation}
1<\beta-\gamma\leq2 \label{bkc}%
\end{equation}

Further, the difference $\beta-\gamma$ is decreasing as a function of $k$ with
$\lim_{k\rightarrow\infty}(\beta-\gamma)=1$.
\end{lemma}

\begin{proof}
Note that $\beta-\gamma\leq2$ if and only if $\gamma\geq\beta-2$ and this
inequality is true if and only if%
\[
k-\beta=\gamma^{2}\geq(\beta-2)^{2}=\beta^{2}-4\beta+4
\]

Since $\beta$ is a fixed point, $\beta^{2}=k+\beta$ so the above inequality is
true if and only if%
\[
-\beta\geq-3\beta+4
\]
which is true if and only if $\beta\geq2$. By \ref{bk} this
is the case if $k\geq2$.

Similarly, $1<\beta-\gamma$ if and only if%
\[
k-\beta=\gamma^{2}<(\beta-1)^{2}=\beta^{2}-2\beta+1=k-\beta+1
\]
which is obviously true.

The decreasing nature of $\beta-\gamma$ as a function of $k$ may be
established by straightforward calculation using derivatives. Now, take the
limit:%
\begin{align*}
\lim_{k\rightarrow\infty}(\beta-\gamma)  &  =\lim_{k\rightarrow\infty}\left(
\frac{1+\sqrt{4k+1}}{2}-\sqrt{k-\frac{1+\sqrt{4k+1}}{2}}\right) \\
&  =\frac{1}{2}+\lim_{k\rightarrow\infty}\left(  \sqrt{k+1/4}-\sqrt
{k-1/2-\sqrt{k+1/4}}\right)
\end{align*}

To calculate the limit of the indeterminate form we multiply and divide by the
conjugate to get:%
\[
\lim_{k\rightarrow\infty}(\beta-\gamma)=\frac{1}{2}+\lim_{k\rightarrow\infty
}\frac{3/4+\sqrt{k+1/4}}{\sqrt{k+1/4}+\sqrt{k-1/2-\sqrt{k+1/4}}}%
\]

The limit may now be determined as follows:%
\begin{align*}
\lim_{k\rightarrow\infty}(\beta-\gamma)  &  =\frac{1}{2}+\lim_{k\rightarrow
\infty}\frac{3/[4\sqrt{k+1/4}]+1}{1+\sqrt{(k-1/2)/(k+1/4)-\sqrt{k+1/4}%
/(k+1/4)}}\\
&  =\frac{1}{2}+\frac{1}{1+\sqrt{1-0}}=1
\end{align*}

This concludes the proof.
\end{proof}

\medskip

It is also useful to write (\ref{bkc}) as follows:%
\begin{equation}
\beta-2\leq\gamma<\beta-1 \label{bkc1}%
\end{equation}

Because the length of $[\gamma,\beta]$ is larger than 1 it contains at least
one integer for every $k\geq2$. I now show that for certain values of $k$ the
interval $[\gamma,\beta]$ contains\textit{ }two distinct integers. For the
exceptional value $k=2$ we have $\gamma=0$ and $\beta=2$ so $[\gamma
,\beta]=[0,2]$ contains three distinct integers.

\begin{lemma}
\label{L4}Assume that $k\geq2$ (so that $\gamma$ is real).

(a) If $k=j(j+1)$ or $k=j(j+1)+1$ for some positive integer $j$ then%
\begin{equation}
j,\,j+1\in\lbrack\gamma,\beta] \label{jj1}%
\end{equation}

(b) If $k\not =j(j+1),\,j(j+1)+1$ for all positive integers $j$ then
$[\gamma,\beta]$ contains exactly one positive integer that is different from
both $\gamma$ and $\beta$.
\end{lemma}

\begin{proof}
(a) Note that if $k=j(j+1)$ for some integer $j$ then
\[
\beta=\frac{1+\sqrt{1+4j(j+1)}}{2}=\frac{1+2j+1}{2}=j+1
\]
and
\[
\gamma=\sqrt{k-\beta}=\sqrt{j^{2}-1}\,<j
\]
so (\ref{jj1}) is true for $k=j(j+1)$. Next, for $k=j(j+1)+1=j^{2}+j+1$
\[
\beta=\frac{1+\sqrt{1+4j^{2}+4j+4}}{2}=\frac{1+\sqrt{(1+2j)^{2}+4}}{2}>j+1
\]
so%
\[
\gamma=\sqrt{j^{2}+j+1-\beta}<\sqrt{j^{2}+j+1-(j+1)}=j
\]

It follows that (\ref{jj1}) is true for $k=j(j+1)+1$ also and the proof of (a)
is complete.

(b) Let $\beta_{k}$ be the fixed point of $f(x)=x^{2}-k$ and define
$\gamma_{k}=\sqrt{k-\beta_{k}}$. For all non-negative $j$ define%
\[
k_{j}=j(j+1)
\]

Note that the sequence of (even) integers $k_{j}$ is increasing as a function
of $j$ and%
\[
k_{j+1}=(j+1)(j+2)=k_{j}+2j+2
\]

Therefore, for each fixed value of $j$
\[
k_{j}<j(j+1)+i<k_{j+1},\qquad i=1,2,\ldots,2j+1
\]

By Part (a) we know that $[\gamma_{k_{j}+1},\beta_{k_{j}+1}]$ contains both
$j$ and $j+1$. Further, since $\beta_{k}$ increases with $k$ and the smallest
value of $k$ where $\beta_{k}=j+2$ is $k_{j+1}=(j+1)(j+2)$, it follows that
$j+2\not \in \lbrack\gamma_{k_{j}+i},\beta_{k_{j}+i}]$ for $i=1,2,\ldots,2j+1$.

Now we show that if $i\geq2$ then $[\gamma_{k_{j}+i},\beta_{k_{j}+i}]$
contains only one integer: $j+1$.

To prove this claim, first note that $j+1<\beta
_{k_{j}+i}$ for all $i$ because $\beta_{k}$ is an increasing function of $k$.
Similarly, $\gamma_{k}$ increases with $k$ and $\gamma_{k_{j}+2j+1}<j+1$ for
$i=2j+1$ because after squaring it we obtain%
\[
k_{j}+2j+1-\frac{1+\sqrt{1+4(k_{j}+2j+1)}}{2}<(j+1)^{2}%
\]

Substituting for $k_{j}$ and doing little algebra we see that this inequality
if and only if%
\[
2j-1<\sqrt{(2j+1)^{2}+8j+4}%
\]
which is obviously true. So $j+1\in\lbrack\gamma_{k_{j}+i},\beta_{k_{j}+i}]$
for $i=2,\ldots,2j+1$. On the other hand, for $i=2$ we have $\gamma_{k_{j}%
+2}>j$ if and only if%
%\begin{align*}
\[k_{j}+2-\frac{1+\sqrt{1+4(k_{j}+2)}}{2} >j^{2}
\]
and this inequality holds if and only if
\[
2j+3 >\sqrt{(2j+3)^{2}-8j}
\]

Since the last inequality is true for $j\geq1$ our claim is justified.
Further, $\gamma_{k}$ increases with $k$ which implies that $j\not \in
\lbrack\gamma_{k_{j}+i},\beta_{k_{j}+i}]$ for $i=2,\ldots,2j+1$ and the proof
is complete.
\end{proof}

\medskip

Figure \ref{ro-v-r} illustrates the above lemma. The upper curve is $\beta
_{k}$ and the lower is $\gamma_{k}$. The dashed curve shows $\beta_{k}-1$. 

%FIGURE ro-v-r

\begin{figure}[h] % float placement: (h)ere, page (t)op, page (b)ottom, other (p)age
  \centering
  % file name: C:/Documents and Settings/Sue/My Documents/HS/HS-Math/ro-v-r.PNG
  \includegraphics[width=4.09in,height=2.1in,keepaspectratio]{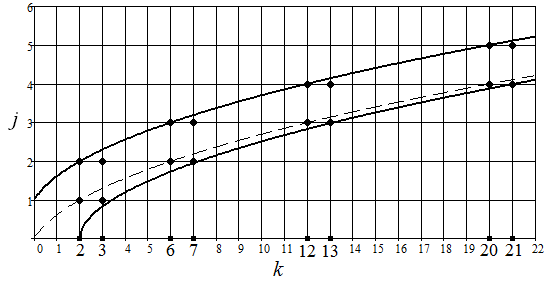}
  \caption{Bounding curves for integer solutions}
  \label{ro-v-r}
\end{figure}

The special values of $k$ where the interval $[c_{k},b_{k}]$ contains two points
are highlighted by dots and by numbers in larger font.

\begin{theorem}
\label{m}Every sequence $r_{n}$ of integers that is an orbit of (\ref{hd}) 
must satisfy one of the following conditions:

(a) If $k=j(j+1)$ for some integer $j$ then $r_{n}$ is one of two constant
sequences, $r_{n}=j+1$ or $r_{n}=-j$;

(b) If $k=j(j+1)+1$ for some integer $j$ then $r_{n}$ is the 2-cycle
$-(j+1),\,j$ for $n\geq1$;

(c) If the value of $k$ is not as given in (a) or (b) then$\ r_{n}$ diverges
to infinity. In particular, there are no p-cycles for $p>2$.
\end{theorem}

\begin{proof}
(a) This was established in Lemma \ref{L4}.

(b) This follows from Lemma \ref{L4}(a) and the observation that%
\begin{align*}
f(\pm j)  &  =j^{2}-[j(j+1)+1]=-j-1=-(j+1)\\
f(\pm(j+1))  &  =(j+1)^{2}-[j(j+1)+1]=j
\end{align*}

Notice that the orbit $j\rightarrow-(j+1)\rightarrow j\rightarrow
-(j+1)\rightarrow\cdots$ is the 2-cycle and each of the remaining two points
in the set $[-\beta_{k+1},-\gamma_{k+1}]\cup\lbrack\gamma_{k+1},\beta_{k+1}]$
is mapped to either $j$ or $-(j+1)$ with $k=j(j+1)$.

(c) If the value of $k$ is not as given in (a) or (b) above, i.e. if
$k\not =j(j+1),j(j+1)+1$ then by Lemma \ref{L4}(b) the set $[-\beta
_{k},-\gamma_{k}]\cup\lbrack\gamma_{k},\beta_{k}]$ contains only two points,
say, $j\in\lbrack\gamma_{k},\beta_{k}]$ and thus $-j\in\lbrack-\beta
_{k},-\gamma_{k}]$ with $j\not =\gamma_{k},\beta_{k}$ (therefore, $j$ is not a
fixed point of $f$). Further, $f(j)=j^{2}-k=-j$ if and only if $k=j(j+1)$
which is ruled out by assumption. Thus $f(j)\not =-j$ which means that $f(j)$
is not in the set $[-\beta_{k},-\gamma_{k}]\cup\lbrack\gamma_{k},\beta_{k}]$.
Thus, by Lemma \ref{L1} there are no bounded solutions in this case and 
therefore, no cycles either.
\end{proof}

\section{Extension to the general quadratic map}

In this section we discuss how to extend the results of the previous section
to the general quadratic function%
\[
Q(x)=ax^{2}+bx+c\qquad a,b,c\in\mathbb{Z\quad}a\not =0
\]

In this case $Q:\mathbb{Z\rightarrow Z}$ is a mapping of the integers and the
recursion
\begin{equation}
x_{n+1}=Q(x_{n})=ax_{n}^{2}+bx_{n}+c,\quad x_{0}\in\mathbb{Z} \label{qr}%
\end{equation}
generates integer sequences.

The key observation about $Q$ is that unlike polynomials of degree 3 or
greater, the general quadratic function $Q$ is conjugate to the special case%
\[
f(x)=x^{2}-q
\]
where $q$ is a rational number. The only difference between this mapping and
the one we studied in the previous section is that $q$ is not an integer if
$b$ is odd. Many of the results of the previous section apply to rational $q$
as well so we simply need to point out how to make the connection. We start
with the following lemma.

\begin{lemma}
\label{tr}Let $a_{i},b_{i},c_{i}$ for $i=1,2$ be fixed real numbers. If
$a_{1},a_{2}\not =0$ and
\begin{equation}
a_{1}(b_{1}+c_{1})=a_{2}(b_{2}+c_{2}) \label{cfs}%
\end{equation}
then\textit{ the mappings }$f_{i}(x)=a_{i}(x+b_{i})^{2}+c_{i}$, $i=1,2$
\textit{are topologically conjugate}; that is, there is a homeomorphism $h$
such that%
\begin{equation}
h\circ f_{1}=f_{2}\circ h \label{fh}%
\end{equation}

In fact,
\begin{equation}
h(x)=\frac{a_{1}}{a_{2}}x+\frac{a_{1}b_{1}-a_{2}b_{2}}{a_{2}} \label{h}%
\end{equation}

\end{lemma}

\begin{proof}
Consider the function $h(x)=\alpha x+\beta$ which is a homeomorphism of the
set of real numbers if $\alpha\not =0$. The equality in (\ref{fh}) holds if
and only if
\begin{align*}
\alpha a_{1}(x+b_{1})^{2}+\alpha c_{1}+\beta &  =a_{2}(\alpha x+\beta
+b_{2})^{2}+c_{2}\\
\alpha a_{1}(x+b_{1})^{2}+\alpha c_{1}+\beta &  =a_{2}\alpha^{2}\left(
x+\frac{b_{2}+\beta}{\alpha}\right)  ^{2}+c_{2}%
\end{align*}

The last equality holds if $\alpha,\beta$ can be chosen so that%
\begin{equation}
\alpha a_{1}=a_{2}\alpha^{2},\quad b_{1}=\frac{b_{2}+\beta}{\alpha}%
,\quad\alpha c_{1}+\beta=c_{2} \label{cof}%
\end{equation}

The first two of the above equalities give us%
\[
\alpha=\frac{a_{1}}{a_{2}},\quad\beta=\frac{a_{1}b_{1}-a_{2}b_{2}}{a_{2}}%
\]

Further, $\alpha,\beta$ must satisfiy the third equality in (\ref{cof}):%
\begin{align*}
\alpha c_{1}+\beta &  =c_{2}\\
a_{1}c_{1}+a_{1}b_{1}-a_{2}b_{2}  &  =a_{2}c_{2}%
\end{align*}

The last equality is equivalent to (\ref{cfs}).
\end{proof}

\medskip

Next, observe that since%
\[
Q(x)=a\left(  x^{2}+\frac{b}{a}x\right)  +c=a\left(  x+\frac{b}{2a}\right)
^{2}+c-\frac{b^{2}}{4a}%
\]
the following corollary of Lemma \ref{tr} is obtained by setting%
\[
a_{1}=a,\ b_{1}=\frac{b}{2a},\ c_{1}=c-\frac{b^{2}}{4a}\quad\text{and\quad
}a_{2}=1,\ b_{2}=0
\]
in (\ref{cfs}) and using the conjugate map $h$.

\begin{lemma}
The quadratic function $Q(x)$ is topologically conjugate to the translation%
\begin{equation}
f(x)=x^{2}-q,\quad q=\frac{b^{2}}{4}-\frac{b}{2}-ac=\frac{b(b-2)}{4}-ac
\label{q}%
\end{equation}

Every orbit $r_{n}$ of (\ref{q}) uniquely corresponds via the homeomorphism
$h$ to an orbit $s_{n}$ of (\ref{qr}) as follows:
\[
r_{n}=as_{n}+\frac{b}{2}%
\]

Equivalently,%
\begin{equation}
s_{n}=\frac{r_{n}}{a}-\frac{b}{2a} \label{sr}%
\end{equation}

\end{lemma}

The main issue now is to show that there are rational orbits $r_{n}$ of%
\begin{equation}
x_{n+1}=x_{n}^{2}-q \label{qe}%
\end{equation}
that yield all the integer orbits $s_{n}$ of (\ref{qr}) via (\ref{sr}).

We begin with the observation that if $b$ is even then $q$ in (\ref{q}) is an
integer so we may apply Theorem \ref{m} directly to the quadratic function
$f(x)$ in (\ref{q}) and obtain the next corollary about the integer orbits of
(\ref{qr}).

\begin{corollary}
\label{ab}Assume that $b$ is an even integer in (\ref{qr}).

(a) There is at least one integer fixed point for (\ref{qr}) if
\[
\frac{b(b-2)}{4}-ac=j(j+1)
\]
for some integer $j$. The integer fixed point is one of the following (both of
them if $a=\pm1$)
\[
\frac{j}{a}-\frac{b-2}{2a},\quad-\frac{j}{a}-\frac{b}{2a}%
\]

(b) There is an integral 2-cycle for (\ref{qr}) if
\[
\frac{b(b-2)}{4}-ac=j(j+1)+1
\]
for some integer $j$. This 2-cycle is%
\[
\frac{j}{a}-\frac{b}{2a},\quad-\frac{j}{a}-\frac{b+2}{2a}%
\]

(c) If $ac$ is not as given in (a) or (b) then every integer orbit of
(\ref{qr}) increases to $\infty$ if $a>0$ or decreases to $-\infty$ if $a<0$.
In particular, (\ref{qr}) has no integer p-cycles if $p\geq3$.
\end{corollary}

\medskip

Note that in the special case where $a=1$ and $b=0$ the above corollary
reduces to Theorem \ref{m} (with $k=-c$).

To illustrate Corollary \ref{ab} with an example let $j$ be any positive
integer and consider%
\begin{equation}
x_{n+1}=x_{n}^{2}+2x_{n}-l \label{gex1}%
\end{equation}
where $a=1$, $b=2$ and $c=-l$. The recursion in (\ref{gex1}) has a pair of
integer fixed points $j$ and $-j-1$ if $l=j(j+1)$ and it has an integral
2-cycle $j-1$, $-j-2$ if $l=j(j+1)+1$. There are no other cycles of
(\ref{gex1}) for any value of $l$.

For the equation%
\[
x_{n+1}=-2x_{n}^{2}+2x_{n}+1
\]
we have $a=-2$, $b=2$ and $c=1$. With $ac=-2$, Part (a) of Corollary \ref{ab}
holds if $j=1$; of the two fixed points%
\[
\frac{j}{a}-\frac{b-2}{2a}=-\frac{1}{2},\quad-\frac{j}{a}-\frac{b}{2a}=1
\]
only one is an integer. There are no proper cycles in this case.

\medskip

If $b$ is \textit{odd} then Theorem \ref{m} is not applicable, but a modified
form of Corollary \ref{ab} holds. The key observation is the following:

\begin{quote}
\textit{If }$a,b,c$\textit{ have integer values in (\ref{q}) then }%
$4q$\textit{ is an integer.}
\end{quote}

In fact,%
\[
4q=b^{2}-2b-4ac=(b-1)^{2}-1-4ac
\]

From this equality we obtain%
\[
1+4q=(b-1)^{2}-4ac
\]
which is the discriminant of the fixed point equations for both $Q$ and its
conjugate $f$. Indeed, the fixed points of $Q$ are the solutions of $Q(x)=x$,
i.e.%
\[
ax^{2}+(b-1)x+c=0
\]
which yields
\begin{equation}
\frac{-(b-1)\pm\sqrt{(b-1)^{2}-4ac}}{2a} \label{fpq}%
\end{equation}
while from $f(x)=x$, i.e.
\[
x^{2}-x-q=0
\]
we obtain%
\begin{equation}
\frac{1\pm\sqrt{1+4q}}{2} \label{fpc}%
\end{equation}

In order that the numbers in (\ref{fpq}) and (\ref{fpc}) be rational it is
necessary that under the square roots we have perfect squares.

\medskip

Now, suppose that $b$ is odd. Then from (\ref{fpq}) we obtain integers if and
only if $(b-1)^{2}-4ac$ is the square of an \textit{even} integer, i.e.%
\[
(b-1)^{2}-4ac=(2m)^{2}%
\]

The last equation may be written as%
\begin{equation}
ac=\left(  \frac{b-1}{2}\right)  ^{2}-m^{2} \label{cbo}%
\end{equation}

Thus, when $b$ is odd $Q(x)$ has integer fixed points if the product $ac$ is a
number of the above type for some integer $m$. This is how (\ref{cbo})
modifies Part (a) of Corollary \ref{ab} when $b$ is odd. For example, the
quadratic recursion%
\begin{equation}
x_{n+1}=x_{n}^{2}+x_{n}-1 \label{gex0}%
\end{equation}
with $a=b=1$ and $c=-1$ gives
\[
\left(  \frac{b-1}{2}\right)  ^{2}-1=-1=ac
\]
so $m=1$. Indeed, (\ref{gex0}) has a pair of integer fixed points $\pm1$.

\medskip

To extend these observations to cycles with lengths larger than 1 we consider
$f(x)=x^{2}-q$ and the fixed points in (\ref{fpc}). Using notation analogous
to what we previsously discussed for the case of integer $q$, define%
\[
B_{q}=\frac{1+\sqrt{1+4q}}{2},\qquad C_{q}=\sqrt{q-B_{q}}%
\]

These are the same as the earlier parameters $b_{k}$ and $c_{k}$. In fact, if
we think of $k$ (or $q$) as \textit{real} numbers then they are indeed the
same functions but now we check their values for rational $q$. Notice that
\[
q=\frac{b^{2}}{4}-\frac{b}{2}+\frac{1}{4}-ac-\frac{1}{4}=\left(  \frac{b-1}%
{2}\right)  ^{2}-ac-\frac{1}{4}%
\]

Looking back at Figure \ref{ro-v-r}, \textit{when }$b$\textit{ is odd, we
check the region between the two curves at integer values less 1/4} on the
horizontal axis; that is, at $k-1/4$ rather than at integers $k$.

With this in mind, if we set $q=k-1/4$ (or $1+4q=k$) in the square root in
$B_{q}$ then we obtain%
\[
\sqrt{1+4q}=\sqrt{1+4\left(  k-\frac{1}{4}\right)  }=2\sqrt{k}%
\]
which is rational (in fact, integral) if and only if $k=j^{2}$ is a perfect
square. This gives%
\[
B_{q}=\frac{1+2\sqrt{k}}{2}=\frac{1}{2}+j
\]

Therefore, if $q=j^{2}-1/4$ where $j$ is an integer then \textit{the fixed
point }$B_{q}$\textit{ has an integer value plus 1/2}. Figure \ref{ro-v-r2}
illustrates this relationship.

%FIGURE ro-v-r2

\begin{figure}[h] % float placement: (h)ere, page (t)op, page (b)ottom, other (p)age
  \centering
  % file name: C:/Documents and Settings/Hassan/My Documents/HS-Math/ro-v-r2.PNG
  \includegraphics[width=3.95in,height=1.99in,keepaspectratio]{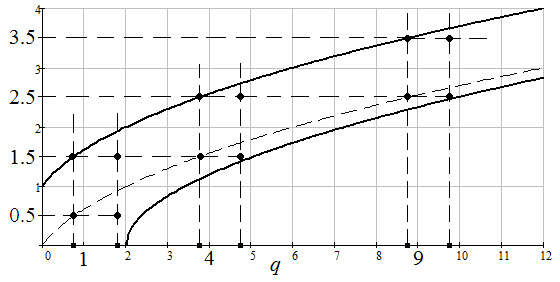}
  \caption{Bounding curves, odd $b$}
  \label{ro-v-r2}
\end{figure}

In Figure \ref{ro-v-r2} we see that for the \textquotedblleft consecutive"
values $j^{2}-1/4$ and $j^{2}-1/4+1$ a square of side 1 fits in the region
between the curves $B_{q}$ and $C_{q}$ just like the earlier case where $q$
was integral. The second fixed point of $Q(x)$ is
\[
A_{q}=\frac{1}{2}-j
\]
since $A_{q}B_{q}=-q$. Note that $A_{q}$ is in the mirror image of
$[C_{q},B_{q}]$, i.e. the interval $[-B_{q},C_{q}].$ Its negative
$-A_{q}=j-1/2$ is in $[C_{q},B_{q}]$ and this is the other point that we see
directly below $B_{q}$ in Figure \ref{ro-v-r2}.

\medskip

Earlier, in Figure \ref{ro-v-r} we saw that the 2-cycles occurred at the value
of $k$ next to the one that produced the fixed points. A similar situation
appears in Figure \ref{ro-v-r2}; the values $q=j^{2}-1/4+1$ for $q=19/4$
($j=2$) and $q=39/4$ ($j=3$) are shown. These are the 2-cycle candidates and
we need only verify this.

Note that the top points at the numbers $q=j^{2}-1/4+1$ are $\alpha=j+1/2$. If
we set $\beta=f(\alpha)$ then%
\[
\beta=\left(  j+\frac{1}{2}\right)  ^{2}-\left(  j^{2}+\frac{3}{4}\right)
=j-\frac{1}{2}%
\]
and%
\[
f(\beta)=\left(  j-\frac{1}{2}\right)  ^{2}-\left(  j^{2}+\frac{3}{4}\right)
=-\left(  j+\frac{1}{2}\right)  =-\alpha
\]

Since $f(\alpha)=f(-\alpha)=\beta$ and $f(\beta)=f(-\beta)=-\alpha$ we see
that $-\alpha,\beta,-\alpha,\beta,\ldots$ is indeed a 2-cycle in the set
\begin{equation}
\lbrack-B_{q},C_{q}]\cup\lbrack C_{q},B_{q}] \label{is}%
\end{equation}

This gives us the proper modification of Part (b) Corollary \ref{ab} when $b$
is odd. Finally, the occurrence of \textit{p}-cycles for $p\geq3$ is
prohibited because it is impossible to fit enough \textquotedblleft integer
plus half" points in the set in (\ref{is}) for each value of $q$.

We summarize these facts in the following.

\begin{corollary}
Assume that $b$ is an odd integer in (\ref{qr}).

(a) There is at least one integer fixed point for (\ref{qr}) if
\[
\left(  \frac{b-1}{2}\right)  ^{2}-ac=j^{2}%
\]
for some integer $j$. The integer fixed point is one of the following (both if
$a=\pm1$)
\[
\frac{j}{a}-\frac{b-1}{2a},\quad-\frac{j}{a}-\frac{b-1}{2a}%
\]

(b) There is an integer 2-cycle for (\ref{qr}) if
\[
\left(  \frac{b-1}{2}\right)  ^{2}-ac=j^{2}+1
\]
for some integer $j$. This 2-cycle is%
\[
-\frac{j}{a}-\frac{b+1}{2a},\quad\frac{j}{a}-\frac{b+1}{2a}%
\]

(c) If $ac$ is not as given in (a) or (b) then every integer orbit of
(\ref{qr}) increases to $\infty$ if $a>0$ or decreases to $-\infty$ if $a<0$.
In particular, (\ref{qr}) has no integer p-cycles if $p\geq3$.
\end{corollary}

\medskip

For instance, consider the quadratic equation%
\begin{equation}
x_{n+1}=x_{n}^{2}+x_{n}-2 \label{gex}%
\end{equation}
with $a=b=1$ and $c=-2$ satisfies the conditions of the above corollary with
$j=1$ since%
\[
\left(  \frac{b-1}{2}\right)  ^{2}-ac=2=j^{2}+1
\]

So there is a 2-cycle whose points are%
\[
-\frac{j}{a}-\frac{b+1}{2a}=-1-1=-2,\quad\frac{j}{a}-\frac{b+1}{2a}=1-1=0
\]

The periodic integer solution of (\ref{gex}) is $-2,0,-2,0,\ldots$ which can
be easily verified by direct substitution into (\ref{gex}).

\medskip

To summarize, we have seen that regarding the existence of periodic integral orbits 
the case $m=2$ is different from all larger, even values of $m$. Another important
difference between these cases is the fact that the general $m$-th degree
polynomial for $m\geq3$ is not linearly conjugate to the simple translation
$x^{m}-k$. 

It is an open question as to whether there are higher degree polynomials whose
iteration generates nontrivial cycles for certain values of integer
coefficients. 

\medskip

Polynomials with \textit{rational} coefficients may well have
\textit{integral} orbits. This is ensured by restricting the coefficient of
the linear term and also the constant term. Consider%
\[
x_{n+1}=a_{m}x_{n}^{m}+a_{m-1}x_{n}^{m-1}+\cdots+a_{2}x_{n}^{2}+a_{1}%
x_{n}+a_{0}%
\]
where $a_{i}$ is rational for $i=0,1,2,\ldots m$ and $a_{m}\not =0$. Assume
that $a_{1}\in\mathbb{Z}$ and let
\[
l=\operatorname{lcm}(d_{2},\ldots,d_{m})
\]
where $d_{i}$ is the denominator of $a_{i}$ for $i=2,\ldots m$. If%
\[
\frac{a_{0}}{l}\in\mathbb{Z}%
\]
then the polynomial function%
\[
f(x)=a_{m}x^{m}+a_{m-1}x^{m-1}+\cdots+a_{2}x^{2}+a_{1}x+a_{0}%
\]
maps the ideal $l\mathbb{Z}$ into itself so if we choose $x_{0}\in
l\mathbb{Z}$ we ensure that $x_{n}\in l\mathbb{Z\subset Z}$ for all $n$.

\medskip

Going further, it would be interesting to characterize possible orbits of 
(\ref{hd}) in the set of all rational numbers.

Alternatively, we may consider the orbits of (\ref{hd}) in finite rings such
as $\mathbb{Z}_{m}$ of integers modulo $m$. Note that all orbits of (\ref{hd})
are necessarily eventually periodic in a finite ring. An interesting question
in this context is what the \textit{maximum length} of a cycle is for a given
$m$.

%\bigskip

\end{document}